\documentclass[12pt]{article}

\usepackage{amsmath, amsthm, amsfonts, geometry}
\newcommand{\Mod}[1]{\ (\text{mod}\ #1)}

\newtheorem{theorem}{Theorem}
\newtheorem{cor}[theorem]{Corollary}

\newtheorem{rem}[theorem]{Remark}
\newtheorem{prb}[theorem]{Problem}
\title{The number of dominating $k$-sets of paths, cycles and wheels}
\author{Jorge L. Arocha\\
  \small Instituto de Matem\'{a}ticas\\
  \small Universidad Nacional Aut\'{o}noma de M\'{e}xico\\
  \small Ciudad Universitaria, M\'{e}xico D.F. 04510\\
  \small e-mail: arocha@matem.unam.mx\\
  \small and\\
        Bernardo Llano\\
  \small Departamento de Matem\'{a}ticas\\
  \small Universidad Aut\'{o}noma Metropolitana - Iztapalapa\\
  \small San Rafael Atlixco 186, Colonia Vicentina, M\'{e}xico, D.F. 09340 \\
  \small e-mail: llano@xanum.uam.mx}

\begin{document}
\maketitle

\abstract{We give a shorter proof of the recurrence relation for the
domination polynomial $\gamma (P_{n},t)$ and for the number $\gamma
_{k}(P_{n})$ of dominating $k$-sets of the path with $n$ vertices.
For every positive integers $n$ and $k,$ numbers $\gamma
_{k}(P_{n})$ are determined
solving a problem posed by S. Alikhani in CID 2015. Moreover, the
numbers of dominating $k$-sets $\gamma _{k}(C_{n})$ of cycles and
$\gamma _{k}(W_{n})$ of wheels with $n$ vertices are computed.


 }

\section{Introduction and preliminaries}

Let $G=\left( V,E\right) $ be a simple graph and a subset of
vertices $D\subseteq V.$ A set $D$ is a \textit{dominating set} if
for every $y\in V-D$ there exists $x\in D$ such that $\left\{
x,y\right\} \in E.$ For every $x\in V$, let $N\left( x\right) $
denote the neighborhood of $x$ and $N\left[ x\right] =N\left(
x\right) \cup \left\{ x\right\}$ is the closed neighborhood of $x.$
Using this notation, $D\subseteq V$ is dominating set if $N\left[
y\right] \cap D\neq \emptyset $ for every $y\in V-D.$ The minimum
cardinality of a dominating set in $G$ is the \textit{domination
number} of $G$ and it is usually denoted by $\gamma (G)$.

Let $\mathcal{D}_{G}$ be the set of every dominating set of $G$.
Observe that $\emptyset \in \mathcal{D}_{G}$ and if $D\in
\mathcal{D}_{G}$ and $D\subset D^{\prime }$, then $D^{\prime }\in
\mathcal{D}_{G}.$ In \cite{AroLl}, the so called\textit{\ domination
polynomial} of a graph $G$ was first introduced as follows.

For every graph $G$, we denote by $\gamma _{k}\left( G\right) $ the
\emph{number of dominating sets of cardinality} $k$ (briefly, the
\emph{dominating $k$-sets}) of $G$. The domination polynomial is
define to be
\[
\gamma \left( G,t\right) =\sum_{k=1}^{n}\gamma _{k}\left( G\right)
t^{n-k},
\]
where $n$ is the number of vertices of $G$.

By definition, we set $\gamma _{0}\left( G\right) =0.$ If $\Phi $
denotes the empty graph, then $\gamma \left( \Phi ,t\right) =0$
since $\mathcal{D}_{\Phi }=\{\emptyset \}$. In the aforementioned
paper, we determined the following domination polynomials:

\begin{equation} \label{Complete}
\gamma (K_{n},t) =\sum_{k=1}^{n}{n \choose
k}t^{n-k}=(1+t)^{n}-t^{n},
\end{equation}
where $K_{n}$ is the complete graph with $n$ vertices,

\begin{equation} \label{Union}
\gamma (\cup _{i=1}^{n}G_{i},t)=\gamma (G_{1},t)\gamma
(G_{2},t)\cdots \gamma (G_{n},t),
\end{equation}
where $\cup _{i=1}^{n}G_{i}$ is the disjoint union of the graphs
$G_{i}$ for $1\leq i\leq n$ (\cite{AroLl}, Theorem 3.1 and Corollary
3.2) and

\begin{align} \label{Sum}
\gamma (G+H,t)= \,&\gamma (K_{n+m},t)-t^{m}\left[ \gamma
(K_{n},t)-\gamma (G,t)\right] \nonumber \\
&-t^{n}\left[ \gamma (K_{m},t)-\gamma(H,t)\right] ,
\end{align}
where $G$ and $H$ are graphs with $n$ and $m$ vertices,
respectively, and $G+H$ denotes the sum of $G$ and $H$
(\cite{AroLl}, Theorem 3.3).

Equivalently (see \cite{AP}, \cite{DT} and \cite{KPSTT}), the
domination
polynomial can be defined as%
\[
D\left( G,t\right) =\sum_{k=0}^{n}\gamma _{k}\left( G\right) t^{k}
\]
and therefore, $\gamma \left( G,t\right) =t^{n}D(G,\frac{1}{t})$.

In recent years, the domination polynomial of graphs has received a
lot of attention. In \cite{DT}, the polynomial is related to network
reliability measure for some special service networks. This new
measure was defined by the authors as the domination and the
corresponding domination reliability polynomial is associated. Some
interesting theoretical problems are studied. It is also proved that
computing domination reliability is NP-hard. The book \cite{BT-book}
by Beichelt and Tittmann broadens the study of network reliability.
The connection of reliability with graph theory and combinatorial
analysis is widely displayed. The book is an excellent example of
applying nontrivial graph theoretical tools to solve key problems in
reliability analysis.

The roots of the domination polynomial have recently been studied,
see for example \cite{AAP}, \cite{BT} and \cite{MRB}. Interesting
problems are still open, particularly, those relating the existence
of some special roots of the domination polynomial to specific
properties of the graph.

Let us define the following binary relation between graphs $G$ and
$H$. We say that $G\sim H$ if and only if $\gamma \left( G,t\right)
=\gamma \left( H,t\right) $. Clearly, this is an equivalence
relation. The determination of the equivalence classes of the
quotient set is a hard important problem. A graph $G$ is
\textit{uniquely determined by its domination polynomial} if $\left[
G\right] =\{G\}.$ In \cite{AAP}, it is proved that $C_{n}=\{C_{n}\}$
for $n\equiv 0,2 \Mod{3}$ and $\left[ P_{n}\right] $ contains two
graphs for $n\equiv 0\Mod{3}$ ($P_{n}$ denotes the path with $n$
vertices). It is also conjectured that $C_{n}=\{C_{n}\}$ for
$n\equiv 1\Mod{3}$.

\begin{prb}
Give a characterization of $\left[ P_{n}\right] $ for $n\equiv
1,2\Mod{3}$.
\end{prb}

In \cite{BA-MP}, the authors characterize those complete $r$-partite graphs
that are uniquely determined by their domination polynomials (see Theorem 2).

The computation of the equivalence classes of almost all other graphs remains an
interesting open question.

The total domination polynomial of a graph is studied in \cite{Dod}.
It is the generating function for the number of total dominating
$k-$sets in a graph $G$. A generalization of the total domination
polynomial, called the trivariate total domination polynomial is
investigated, in particular, this kind of polynomial is studied for
some graph products. There are other variations and generalizations
of the domination polynomial as the independent domination
polynomial or the bipartition polynomial. The already mentioned
paper by Dod and the recent work \cite{DKPT} by Dod, Kotek, Preen
and Tittmann are suitable references for these topics.

For a simple graph $G=(V,E)$, $v\in V$ and $e=\{u,v\}\in E$, we
define the following operations: $G-v$ is the usual \textit{vertex
deletion}; $G-N[v]$ is the \textit{vertex extraction} defined to be
the graph $\bigcup_{u\in N[v]}(G-u)$; $G/v$ is the \textit{vertex
contraction} defined as the graph obtained from $G$ by removing $v$
and adding the edges between any pair of non-adjacent neighbors of
$v$; $G/N(v)$ is the \textit{neighborhood contraction} defined as
the graph defined from $G$ by removing every vertex of $N(v)$ (but
not itself $v$) and adding the edges $\{v,w\}$ for every $w\in
N(N(v))$, where $N(N(v))$ is the \textit{second neighborhood} of
$v$; $G-e$ is the usual \textit{edge deletion}; $G/e$ is the
usual\textit{\ edge contraction} and  $G\dagger e$ is the
\textit{edge extraction} defined to be the graph
$G-\{u,v\}=(G-u)\cup (G-v)$.

One of the main results of \cite{KPSTT} is the following

\begin{theorem} [(\cite{KPSTT}, Theorems 2.4 and 2.5)]
There do not exist rational functions $f_{i},g_{j}\in \mathbb{R}(t)$
with $i\in \{1,2,3,4\}$ and $j\in \{1,2,3\}$ such that for every
graph $G$, $v\in
V $ and $e\in E$, it holds that%
\begin{eqnarray*}
D(G,t)&=&f_{1}D(G-v,t)+f_{2}D(G/v,t)+f_{3}D(G-N[v],t)\\
  &&+f_{4}D(G/N(v),t)\text{ and}\\
D(G,t)&=&g_{1}D(G-e,t)+g_{2}D(G/e,t)+g_{3}D(G\dagger e,t).
\end{eqnarray*}
\end{theorem}

Despite the theorem above states that the domination polynomial
satisfies no linear recurrence relation with the defined operations,
it is possible to defined a very useful recurrence relation using an
special class of oriented graphs obtained from the undirected ones.
This approach was introduced and applied in \cite{AroLl}. We
summarize the main definitions and results that will be used later
in this paper.

Consider the oriented graph $\Gamma =\left( U,A\right) $. We denote
by $N^{-}\left( x\right) $ and $N^{+}\left( x\right) $ the in- and
out-neighborhood of a vertex $x\in U,$ respectively, and
$N^{-}\left[ x \right] =N^{-}(x)\cup \{x\}$ and $N^{+}\left[
x\right] =N^{+}(x)\cup \{x\}$ the respective closed in- and
out-neighborhood of $x$. We say that $D\subseteq U$ is a
\textit{dominating set} of $\Gamma $ if for every $v\in U-D$ there
exists $u\in D$ such that $\left( u,v\right) \in A$, that is
$N^{-}\left[ v\right] \cap D\neq \emptyset .$ From this definition,
if $D$ is a dominating set of $\Gamma $, then there exits $u\in D$
such that $N^{+}(u)\neq \emptyset$. Similarly, the domination
polynomial of an oriented graph $\Gamma $ is defined as
\[
\gamma \left( \Gamma ,t\right) =\sum_{k=1}^{n}\gamma _{k}\left(
\Gamma \right) t^{n-k},
\]
where $\gamma _{k}\left( \Gamma \right) $ denotes the number of
dominating $k $-sets of $\Gamma $.

Let $\Gamma =\left( U_{1},U_{2},A\right) $ a bipartite one-way
oriented graph, that is, a bipartite oriented graph with partite
sets $U_{1}$ and $U_{2}$ such that every arc is oriented from a
vertex of $U_{1}$ to a vertex of $U_{2}$. Observe that if $D$ is a
dominating set of a bipartite one-way oriented graph $\Gamma ,$ then
$D\subseteq U_{1}$ by definition.

\begin{rem}
If $U_{1}=\emptyset $, then $\gamma \left( \Gamma ,t\right) =0$
(there is no
dominating set) and if $U_{2}=\emptyset $, then by convention we define $%
\gamma \left( \Gamma ,t\right) =\left( 1+t\right) ^{\left\vert
U_{1}\right\vert }$. \label{sp-cases}
\end{rem}

We recall that if $\Gamma _{1}$ and $\Gamma _{2}$ are bipartite
components (not necessarily connected) then the domination
polynomial is multiplicative respect to the components, that is
$\gamma \left( \Gamma ,t\right) =\gamma \left( \Gamma _{1},t\right)
\ \gamma \left( \Gamma _{2},t\right) $ (compare with the domination
polynomial (\ref{Union}) for the disjoint union of graphs). The
oriented graph operations $\Gamma -i$ and $\Gamma -N^{+}\left[
i\right]$ are analogously defined as the corresponding vertex
deletion and vertex extraction for graphs.

Let $G=(V,E)$ a simple graph. We construct a bipartite one-way
oriented graph $\Gamma _{G}=\left( U_{1},U_{2},A\right) $ from $G$
such that $U_{1}$
and $U_{2}$ are disjoint copies of $V$ and%
\[
A=\left\{ \left( i,i\right) :i\in V\right\} \cup \left\{ \left(
i,j\right) ,\left( j,i\right) :\left\{ i,j\right\} \in E\right\} .
\]

\begin{theorem}[(\cite{AroLl}, Lemma 3.5, Theorem 3.6)]
Let $G$ be a simple graph. Then
\begin{eqnarray}
\gamma \left( G,t\right) &=&\gamma \left( \Gamma _{G},t\right)
\text{ and}
\notag \\
\gamma \left( \Gamma ,t\right) &=&t\gamma \left( \Gamma -i,t\right)
+\gamma \left( \Gamma -N^{+}\left[ i\right] ,t\right)  \label{rec}
\end{eqnarray}
\noindent for every bipartite one-way oriented graph $\Gamma =\left(
U_{1},U_{2},A\right) $ and $i\in U_{1}$. \label{Thm-rec}
\end{theorem}

\begin{cor}
$\gamma \left( G,t\right) =\gamma \left( \Gamma _{G},t\right)
=t\gamma \left( \Gamma _{G}-i,t\right) +\gamma \left( \Gamma
_{G}-N^{+}\left[ i\right] ,t\right) $ for every simple graph $G$.
\end{cor}

In this paper, we solve the following open problem posed at the
Problem Session of CID 2015 by S. Alikhani:

\begin{prb}
Let $P_{n}$ be the path with $n$ vertices. Find an explicit formula
for $\gamma _{k}(P_{n})$, where $k$ and $n$ are positive integers.
\label{Prob-Ali}
\end{prb}

For this purpose, in Section 2 we give a much shorter proof for the
recurrence relation involving the domination polynomials of paths
proved in \cite{AP}.  Using similar tools, we can show the
recurrence relation for the polynomial of cycles. As a consequence,
we give the domination polynomial of wheels. There are two
corollaries following the respective theorems for the recurrence
relation involving the numbers $\gamma _{k}(P_{n})$ and $\gamma
_{k}(C_{n})$. The number of dominating $k$-sets $\gamma _{k}(W_{n})$
for the wheel with $n$ vertices is a consequence of its domination
polynomial depending on the domination polynomial of a cycle with
$n-1$ vertices. In Section 3 we give the explicit formulas for the
number of the dominating $k$-sets of paths, cycles and wheels using
ordinary generating functions of two variables. In particular,
Theorem \ref{Sol} is a solution to Problem \ref{Prob-Ali}.

For the terminology on graph an digraphs used in what follows, see
\cite{BJ-G}.

\section{The domination polynomial of paths, cycles and wheels}

Let $n$ be a positive integer, $[n]=\{1,...,n\}$ and $P_{n}$ the
path with vertex set $[n]$. In \cite{AP}, the following theorem is
proved.

\begin{theorem}[(\cite{AP}, Theorem 3.1)] \label{Ali}
Let $n$ and $k$ be positive integers. Then for every $n\geq 4$ and
$k\geq 2$
\begin{enumerate}
\item[(i)] $\gamma _{k}(P_{n})=\gamma _{k-1}(P_{n-1})+\gamma
_{k-1}(P_{n-2})+\gamma _{k-1}(P_{n-3})$ \textit{and}
\item[(ii)] $D(P_{n},t)=t\ \left[ D(P_{n-1},t)+D(P_{n-2},t)+D(P_{n-3},t)%
\right] $ \textit{with initial conditions } $D(P_{1},t)=t$,
$D(P_{2},t)=t^{2}+2t$ and $D(P_{3},t)=t^{3}+3t^{2}+t$.
\end{enumerate}
\end{theorem}

The long proof of this theorem is based on six lemmas and Theorem
2.7. We give a much shorter proof of this theorem. For this aim, we
define bipartite one-way oriented graphs $I_{m+1,m}$ and $J_{m,m+1}$
for every positive integer $m.$ First, observe that $\Gamma
_{P_{n}}$ is given by
$U_{1}(\Gamma _{P_{n}})=U_{2}(\Gamma _{P_{n}})=[n]$ and%
\[
A(\Gamma _{P_{n}})=\left\{ \left( i,i\right) :i\in \lbrack
n]\right\} \cup \left\{ \left( i,i+1\right) ,(i+1,i):i\in \lbrack
n-1]\right\} .
\]

Let $U_{1}(I_{m+1,m})=[m+1]$, $U_{2}(I_{m+1,m})=[m]$ and
$$A(I_{m+1,m})=A(\Gamma _{P_{m}})\cup \{(m+1,m)\}.$$
Similarly,
$U_{1}(J_{m,m+1})=[m]$, $U_{2}(J_{m,m+1})=[m+1]$ and
$$A(J_{m,m+1})=A(\Gamma _{P_{m}})\cup \{(m,m+1)\}.$$

\begin{theorem}
Let $n$ be a positive integer. Then
\begin{equation*}
\gamma (P_{n},t)=\gamma (P_{n-1},t)+t\ \gamma
(P_{n-2},t)+t^{2}\gamma (P_{n-3},t)
\end{equation*}
\textit{for every} $n\geq 4$ \textit{with initial conditions}
$\gamma (P_{1},t) =1$, $\gamma (P_{2},t) =1+2t$ and $\gamma
(P_{3},t) =1+3t+t^{2}$.
\end{theorem}

\begin{proof}
Let $n\geq 4$. We apply recurrence relation (\ref{rec}) of Theorem
\ref{Thm-rec} to the domination polynomial $\gamma (\Gamma
_{P_{n}},t)$ and $\gamma (I_{n,n-1},t)$ using vertex $n$ of
$U_{1}(\Gamma _{P_{n}})$ and $U_{1}(I_{n,n-1})$, respectively. We
obtain that
\begin{eqnarray*}
\gamma (\Gamma _{P_{n}},t) &=&t\ \gamma (J_{n-1,n},t)+\gamma (I_{n-1,n-2},t) \text{ and} \\
\gamma (I_{n,n-1},t) &=&t\ \gamma (\Gamma
_{P_{n-1}},t)+\gamma(I_{n-1,n-2},t).
\end{eqnarray*}
Applying again (\ref{rec}) of Theorem \ref{Thm-rec} to $\gamma
(J_{n-1,n},t)$ using vertex $n-1$ of $U_{1}(J_{n-1,n})$, it follows
that
\begin{equation*}
\gamma (J_{n-1,n},t)=t\ \gamma (J_{n-2,n-1}\cup J_{0,1},t)+\gamma
(I_{n-2,n-3},t).
\end{equation*}%
Since
\begin{equation*}
\gamma (J_{n-2,n-1}\cup J_{0,1},t)=\gamma (J_{n-2,n-1},t)\ \gamma
(J_{0,1},t)
\end{equation*}
and by Remark \ref{sp-cases}, $\gamma (J_{0,1},t)=0$
($U_{1}(J_{0,1})=\emptyset )$, we have that $\gamma (J_{n-2,n-1}\cup
J_{0,1},t)=0$ and hence $\gamma (J_{n-1,n},t)=\gamma
(I_{n-2,n-3},t).$ Then
\begin{eqnarray*}
\gamma (\Gamma _{P_{n}},t) &=&\gamma (I_{n-1,n-2},t)+t\ \gamma
(I_{n-2,n-3},t)\text{ and} \\
\gamma (I_{n,n-1},t) &=&t\ \left[ \gamma (I_{n-2,n-3},t)+t\ \gamma
(I_{n-3,n-4},t)\right] +\gamma (I_{n-1,n-2},t) \\
&=&\gamma (I_{n-1,n-2},t)+t\ \gamma (I_{n-2,n-3},t)+t^{2}\gamma
(I_{n-3,n-4},t).
\end{eqnarray*}
Since $n\geq 4$, the last recurrence relation for $\gamma
(I_{n,n-1},t)$ is true with initial conditions $\gamma
(I_{1,0},t)=1+t$ (by Remark \ref{sp-cases}), $\gamma
(I_{2,1},t)=1+2t$ and $\gamma (I_{3,2},t)=1+3t+2t^{2}$. Finally,
\begin{eqnarray*}
\gamma (P_{n},t) &=&\gamma (\Gamma _{P_{n}},t)=\gamma
(I_{n-1,n-2},t)+t\
\gamma (I_{n-2,n-3},t) \\
&=&\gamma (I_{n-2,n-3},t)+t\ \gamma (I_{n-3,n-4},t)+t^{2}\gamma
(I_{n-4,n-5},t) \\
&&+t\left[ \gamma (I_{n-3,n-4},t)+t\ \gamma
(I_{n-4,n-5},t)+t^{2}\gamma
(I_{n-5,n-6},t)\right] \\
&=&\gamma (I_{n-2,n-3},t)+t\ \gamma (I_{n-3,n-4},t)+t\left[ \gamma
(I_{n-3,n-4},t)+t\ \gamma (I_{n-4,n-5},t)\right] \\
&&+t^{2}\left[ \gamma (I_{n-4,n-5},t)+t\ \gamma (I_{n-5,n-6},t)\right] \\
&=&\gamma (\Gamma _{P_{n-1}},t)+t\ \gamma (\Gamma
_{P_{n-2}},t)+t^{2}\gamma
(\Gamma _{P_{n-3}},t) \\
&=&\gamma (P_{n-1},t)+t\ \gamma (P_{n-2},t)+t^{2}\gamma (P_{n-3},t),
\end{eqnarray*}%
which proves the theorem.
\end{proof}

Recalling that $\gamma \left( P_{n},t\right)
=t^{n}D(P_{n},\frac{1}{t})$, we obtain the recurrence relation (ii)
of Theorem \ref{Ali}.

From this theorem and using the definition of the domination
polynomial we have that
\begin{eqnarray*}
\sum_{k=1}^{n}\gamma _{k}\left( P_{n}\right) t^{n-k}
&=&\sum_{k=1}^{n-1}\gamma _{k}\left( P_{n-1}\right)
t^{n-1-k}+t\sum_{k=1}^{n-2}\gamma _{k}\left( P_{n-2}\right) t^{n-2-k} \\
&&+t^{2}\sum_{k=1}^{n-3}\gamma _{k}\left( P_{n-3}\right) t^{n-3-k}
\end{eqnarray*}
\begin{equation*}
=\sum_{k=1}^{n-1}\gamma _{k}\left( P_{n-1}\right)
t^{n-1-k}+\sum_{k=1}^{n-2}\gamma _{k}\left( P_{n-2}\right)
t^{n-1-k}+\sum_{k=1}^{n-3}\gamma _{k}\left( P_{n-3}\right) t^{n-1-k}
\end{equation*}
\begin{eqnarray*}
&=&\sum_{k=2}^{n-2}\left[ \gamma _{k-1}(P_{n-1})+\gamma
_{k-1}(P_{n-2})+\gamma _{k-1}(P_{n-3})\right] t^{n-k} \\
&&+\gamma _{n-2}(P_{n-1})\ t+\gamma _{n-1}(P_{n-1})+\gamma
_{n-2}(P_{n-2})t
\end{eqnarray*}
\begin{equation*}
=\sum_{k=2}^{n-2}\left[ \gamma _{k-1}(P_{n-1})+\gamma
_{k-1}(P_{n-2})+\gamma _{k-1}(P_{n-3})\right] t^{n-k}+n\ t+1
\end{equation*}
and therefore, we have the recurrence relation (i) of Theorem
\ref{Ali}.

\begin{cor} \label{cor-paths}
$\gamma _{k}(P_{n})=\gamma _{k-1}(P_{n-1})+\gamma
_{k-1}(P_{n-2})+\gamma _{k-1}(P_{n-3})$ for every $n\geq 4$ and $2
\leq k \leq n-2$ with initial conditions $\gamma _{1}(P_{1})=1$,
$\gamma _{1}(P_{2})=2$, $\gamma _{1}(P_{3})=1$, $\gamma
_{2}(P_{2})=1$ and $\gamma _{2}(P_{3})=3$. \label{rec-gamma}
\end{cor}

Notice that $\gamma _{n-1}(P_{n})=n$ and $\gamma _{n}(P_{n})=1$.
\bigskip

Let $C_{n}$ be the cycle with $n$ vertices. We define $C_{1}=K_{1}$,
the cycle with one vertex, and $C_{2}$ is the so called
\textit{dicycle} (two vertices joined by two parallel edges). The
domination polynomials are $\gamma (C_{1},t)=1$ and $\gamma
(C_{2},t)=1+2t$, respectively. If we appropriately apply recurrence
relation (\ref{rec}) of Theorem \ref{Thm-rec} to the domination
polynomial $\gamma (\Gamma _{C_{n}},t)$ and define suitable
bipartite one-way oriented graphs, then it can be analogously proved
the following result for the domination polynomial of cycles (see
Theorem 4.5 of \cite{AP-cycles})

\begin{theorem}
Let $n$ be a positive integer. Then
\begin{equation*}
\gamma (C_{n},t)=\gamma (C_{n-1},t)+t \gamma (C_{n-2},t)+t^{2}\gamma
(C_{n-3},t)
\end{equation*}
for every $n\geq 4$ \textit{with initial conditions} $\gamma
(C_{1},t) =1$, $\gamma (C_{2},t) =1+2t$ and $\gamma (C_{3},t)
=1+3t+3t^{2}$.
\end{theorem}

Similarly as done before for paths, we have the following (Theorem 4.4 of \cite{AP-cycles})

\begin{cor} \label{cor-cycles}
$\gamma _{k}(C_{n})=\gamma _{k-1}(C_{n-1})+\gamma
_{k-1}(C_{n-2})+\gamma _{k-1}(C_{n-3})$ for every $n\geq 4$ and $2
\leq k \leq n-2$ with initial conditions $\gamma _{1}(C_{1})=1$,
$\gamma _{1}(C_{2})=2$, $\gamma _{1}(C_{3})=3$, $\gamma
_{2}(C_{2})=1$ and $\gamma _{2}(C_{3})=3$. \label{rec-cycles}
\end{cor}

Notice that $\gamma _{n-1}(C_{n})=n$ and $\gamma _{n}(C_{n})=1$.
\bigskip

The recurrence relations for $\gamma _{k}(P_{n})$ and $\gamma
_{k}(C_{n})$ of Corollaries \ref{cor-paths} and \ref{cor-cycles} are
generalizations to the well-known Tribonacci numbers defined by the
recurrence relation  $T_{n}=T_{n-1}+T_{n-2}+T_{n-3}$ for every  $n
\geq 3$ with initial conditions  $T_{0}=T_{1}=1$ and $T_{2}=2$.
Originally, the Tribonacci sequence was first discussed by Feinberg
in \cite{F1}. In \cite{F2}, these numbers are expressed as sums of
numbers along diagonal planes of the the so-called Pascal's pyramid,
a natural generalization of the Pascal's triangle.

Later, a closed formula for the Tribonacci numbers was proved by Shannon in \cite{Shannon}.

\begin{theorem}[\cite{Shannon}]
For every nonnegative integer $n$,
\[
T_{n}=\sum_{m=0}^{\left\lfloor \frac{n}{2}\right\rfloor
}\sum_{r=0}^{\left\lfloor \frac{n}{3}\right\rfloor }\binom{n-m-2r}{m+r}%
\binom{m+r}{r}.
\]
\end{theorem}

The identity of this theorem generates sequence A000073 of
\cite{Sloane}. This formula is in some sense "analogous" to the
identities of Theorems \ref{Sol} and \ref{Sol-cycles} for the number
of dominating $k$-sets of paths and cycles, respectively.
\bigskip

Let $W_{n}$ be the wheel with $n$ vertices. Recall that
$W_{n}=K_{1}+C_{n-1}$. Using the identities (\ref{Complete}) and
(\ref{Sum}) for domination polynomials of complete graphs and the
sum of two graphs, respectively, we have the following results.

\begin{theorem}
Let $n$ be a positive integer. Then $\gamma
(W_{n},t)=(1+t)^{n-1}+t\gamma (C_{n-1},t)$ for every $n\geq 4$.
\end{theorem}

\begin{cor}
$\gamma _{k}(W_{n})=\binom{n-1}{n-k}+\gamma _{k}(C_{n-1})$ for every
$n\geq 4 $ and $1 \leq k \leq n$. \label{wheels}
\end{cor}

\section{The number of dominating $k$-sets of paths, cycles and wheels with $%
n$ vertices}

In this section, we explicitly compute the numbers $\gamma
_{k}(P_{n})$ for every positive integers $n$ and $k$ and so, we give
a solution to Problem \ref{Prob-Ali}. By Corollary \ref{rec-gamma},
we know that $\gamma _{k}(P_{n})=\gamma _{k-1}(P_{n-1})+\gamma
_{k-1}(P_{n-2})+\gamma _{k-1}(P_{n-3})$ for every $n\geq 4$ and $2
\leq k \leq n-2$ with initial conditions $\gamma _{1}(P_{1})=1$,
$\gamma _{1}(P_{2})=2$ and $\gamma _{1}(P_{3})=1$. Let us denote
$\gamma _{n,k}=\gamma _{k}(P_{n}).$ We define the ordinary
generating function for the numbers $\gamma _{n,k}=\gamma
_{k}(P_{n})$ as
\begin{equation*}
G(x,y)=\sum_{n\geq 1}\sum_{k\geq 1}\gamma _{n,k}\ x^{n}\ y^{k}.
\end{equation*}

\begin{theorem}
\label{gf-paths}
\begin{equation*}
G(x,y)=\frac{x(1+x)^{2}y}{1-(x+x^{2}+x^{3})y}.
\end{equation*}
\end{theorem}

\begin{proof}
Consider the following identity:%
\begin{eqnarray*}
&&\left[ 1-(x+x^{2}+x^{3})y\right] G(x,y) \\
&=&G(x,y)-x~y~G(x,y)-x^{2}y~G(x,y)-x^{3}y~G(x,y)
\end{eqnarray*}
\begin{eqnarray*}
&=&\sum_{n\geq 1}\sum_{k\geq 1}\gamma _{n,k}\ x^{n}\
y^{k}-\sum_{n\geq
1}\sum_{k\geq 1}\gamma _{n,k}\ x^{n+1}\ y^{k+1} \\
&&-\sum_{n\geq 1}\sum_{k\geq 1}\gamma _{n,k}\ x^{n+2}\
y^{k+1}-\sum_{n\geq 1}\sum_{k\geq 1}\gamma _{n,k}\ x^{n+3}\ y^{k+1}
\end{eqnarray*}
\begin{eqnarray*}
&=&\sum_{n\geq 1}\sum_{k\geq 1}\gamma _{n,k}\ x^{n}\
y^{k}-\sum_{n\geq
2}\sum_{k\geq 2}\gamma _{n-1,k-1}\ x^{n}\ y^{k} \\
&&-\sum_{n\geq 3}\sum_{k\geq 2}\gamma _{n-2,k-1}\ x^{n}\
y^{k}-\sum_{n\geq 4}\sum_{k\geq 2}\gamma _{n-3,k-1}\ x^{n}\ y^{k}
\end{eqnarray*}
\begin{eqnarray*}
&=&\gamma _{1,1}~x~y+\gamma _{2,1}~x^{2}~y+\gamma
_{3,1}~x^{3}~y+\gamma
_{1,2}~x~y^{2}+\gamma _{2,2}~x^{2}~y^{2}+\gamma _{3,2}~x^{3}~y^{2} \\
&&+\sum_{n\geq 4}\sum_{k\geq 2}\gamma _{n,k}\ x^{n}\ y^{k}-\left[
\gamma _{1,1}~x^{2}y^{2}+\gamma _{2,1}~x^{3}~y^{2}\right]
-\sum_{n\geq
4}\sum_{k\geq 2}\gamma _{n-1,k-1}\ x^{n}\ y^{k} \\
&&-\gamma _{1,1}~x^{3}~y^{2}-\sum_{n\geq 4}\sum_{k\geq 2}\gamma
_{n-2,k-1}\ x^{n}\ y^{k}-\sum_{n\geq 4}\sum_{k\geq 2}\gamma
_{n-3,k-1}\ x^{n}\ y^{k}
\end{eqnarray*}
\begin{eqnarray*}
&=&\gamma _{1,1}\left[ x~y-x^{2}~y^{2}-x^{3}~y^{2}\right] +\gamma _{2,1}%
\left[ ~x^{2}~y-x^{3}~y^{2}\right] +\gamma _{3,1}~x^{3}~y \\
&&+\gamma _{1,2}~x~y^{2}+\gamma _{2,2}~x^{2}~y^{2}+\gamma
_{3,2}~x^{3}~y^{2}
\\
&&+\sum_{n\geq 4}\sum_{k\geq 2}\left[ \gamma _{n,k}-\gamma
_{n-1,k-1}-\gamma _{n-2,k-1}-\gamma _{n-3,k-1}\right] \ x^{n}\ y^{k}
\end{eqnarray*}
\begin{eqnarray*}
&=&x~y-x^{2}~y^{2}-x^{3}~y^{2}+2~x^{2}~y-2~x^{3}~y^{2}+~x^{3}~y+x^{2}~y^{2}+3~x^{3}~y^{2}
\\
\allowbreak &=&xy+2x^{2}y+x^{3}y=x(1+x)^{2}y,
\end{eqnarray*}
where $\gamma _{n,k}=\gamma _{k}(P_{n})=0$ if $n<k$ and using
Corollary \ref{rec-gamma}.
\end{proof}

If we determine the formal power series of $G(x,y)$ expanded in
powers of $y$,
we have that%
\begin{eqnarray*}
G(x,y) &=&\frac{x(1+x)^{2}y}{1-(x+x^{2}+x^{3})y} \\
&=&x(1+x)^{2}y\sum_{k\geq 0}(x+x^{2}+x^{3})^{k}y^{k} \\
&=&\sum_{k\geq 1}x(1+x)^{2}(x+x^{2}+x^{3})^{k-1}y^{k} \\
&=&\sum_{k\geq 1}\left( \sum_{n\geq 1}\gamma
_{k}(P_{n})~x^{n}\right) y^{k}
\end{eqnarray*}
and then for every $k\geq 1$%
\begin{eqnarray}
\sum_{n\geq 1}\gamma _{k}(P_{n})~x^{n}
&=&x(1+x)^{2}(x+x^{2}+x^{3})^{k-1}
\notag \\
&=&x^{k}(1+x)^{2}(1+x+x^{2})^{k-1}  \label{poly}
\end{eqnarray}
is a monic polynomial denoted by $g_{k}(x)$ of degree $3k$ such that%
\begin{equation*}
g_{k}(x)=\sum_{n=k}^{3k}\gamma
_{k}(P_{n})~x^{n}=\sum_{t=0}^{2k}\gamma _{k}(P_{k+t})~x^{k+t}
\end{equation*}
since $\gamma _{k}(P_{n})=0$ if $n<k$ and the well-known fact that
$\gamma (P_{n})=\left\lceil \frac{n}{3}\right\rceil $ for every
$n\geq 1$. Observe that $g_{k}(x)$ has $2k+1$ terms,
$g_{k}(x)=x^{2k}g_{k}(\frac{1}{x})$ and so, $\gamma
_{k}(P_{n})=\gamma _{k}(P_{4k-n})$ which means that the polynomial
is symmetric with respect to the ($k+1)$-th term.

\begin{theorem} \label{Sol}
For every $k\geq 1$ and $t\geq 0$, the number of dominating $k$-sets
of the path $P_{k+t}$ is
\begin{equation*}
\gamma _{k}(P_{k+t})=\sum_{m=0}^{\left\lfloor
\frac{t}{2}\right\rfloor +1} \binom{k-1}{t-m}\binom{t-m+2}{m}.
\end{equation*}
\end{theorem}

\begin{proof}
From identity (\ref{poly}), we have that
\begin{equation*}
(1+x+x^{2})^{k-1}=\sum_{l=0}^{k-1}\binom{k-1}{l}(x+x^{2})^{l}=%
\sum_{l=0}^{k-1}\binom{k-1}{l}\sum_{m=0}^{l}\binom{l}{m}x^{l-m}x^{2m}
\end{equation*}
and hence
\begin{equation}
(1+x+x^{2})^{k-1}=\sum_{l=0}^{k-1}\sum_{m=0}^{l}\binom{k-1}{l}\binom{l}{m}%
x^{l+m}.  \label{1xx2}
\end{equation}
Therefore,
\begin{equation*}
g_{k}(x)=(x^{k}+2x^{k+1}+x^{k+2})(1+x+x^{2})^{k-1}
\end{equation*}
\begin{eqnarray*}
&=&x^{k}\sum_{l=0}^{k-1}\binom{k-1}{l}\left( \sum_{m=0}^{l}\binom{l}{m}%
x^{m}\right) x^{l}+2x^{k+1}\sum_{l=0}^{k-1}\binom{k-1}{l}\left(
\sum_{m=0}^{l}\binom{l}{m}x^{m}\right) x^{l} \\
&&+x^{k+2}\sum_{l=0}^{k-1}\binom{k-1}{l}\left( \sum_{m=0}^{l}\binom{l}{m}%
x^{m}\right) x^{l}
\end{eqnarray*}
\begin{eqnarray*}
&=&\sum_{l=0}^{k-1}\sum_{m=0}^{l}\binom{k-1}{l}\binom{l}{m}%
x^{k+l+m}+\sum_{l=0}^{k-1}\sum_{m=0}^{l}2\binom{k-1}{l}\binom{l}{m}%
x^{k+l+m+1} \\
&&+\sum_{l=0}^{k-1}\sum_{m=0}^{l}\binom{k-1}{l}\binom{l}{m}x^{k+l+m+2}
\end{eqnarray*}
\begin{equation*}
=\sum_{l=0}^{k-1}\binom{k-1}{l}x^{k}\sum_{m=0}^{l}\left[ \binom{l}{m}%
x^{l+m}+2\binom{l}{m}x^{l+m+1}+\binom{l}{m}x^{l+m+2}\right]
\end{equation*}
\begin{equation*}
=\sum_{l=0}^{k-1}\binom{k-1}{l}x^{k}\left[ \sum_{m=0}^{l+2}\left( \binom{l}{m%
}+2\binom{l}{m-1}+\binom{l}{m-2}\right) x^{m+l}\right] .
\end{equation*}
Using Pascal's formula,
\begin{eqnarray*}
g_{k}(x) &=&x^{k}\sum_{l=0}^{k-1}\binom{k-1}{l}\left[ \sum_{m=0}^{l+2}\binom{%
l+2}{m}x^{m}\right] x^{l} \\
&=&x^{k}\sum_{l=0}^{k-1}\sum_{m=0}^{l+2}%
\binom{k-1}{l}\binom{l+2}{m}x^{l+m} \\
&=&x^{k}\sum_{t=0}^{2k}\sum_{m=0}^{\left\lfloor \frac{t}{2}\right\rfloor +1}%
\binom{k-1}{t-m}\binom{t-m+2}{m}x^{t} \\
&=&\sum_{t=0}^{2k}\sum_{m=0}^{\left%
\lfloor \frac{t}{2}\right\rfloor +1}\binom{k-1}{t-m}\binom{t-m+2}{m}x^{k+t}.%
\text{ }
\end{eqnarray*}
Hence,%
\begin{equation*}
\gamma _{k}(P_{k+t})=\sum_{m=0}^{\left\lfloor \frac{t}{2}\right\rfloor +1}%
\binom{k-1}{t-m}\binom{t-m+2}{m}
\end{equation*}
which is what had to be proven.
\end{proof}

The result of this theorem is an explicit formula for sequence A212633 of \cite%
{Sloane}.

\begin{cor}
Let $k$ be a positive integer. The number of the dominating sets of
minimum cardinality for paths is
\begin{equation*}
\begin{array}{ll}
\qquad \gamma _{k}(P_{3k})=1 & \text{if }n=3k\text{,} \\
\gamma _{k+1}(P_{3k+1})=\binom{k+2}{2}+k & \text{if }n=3k+1\text{ and} \\
\gamma _{k+1}(P_{3k+2})=k+2 & \text{if }n=3k+2\text{.}
\end{array}
\end{equation*}
\end{cor}

By Corollary \ref{rec-cycles}, we have that $\gamma
_{k}(C_{n})=\gamma _{k-1}(C_{n-1})+\gamma _{k-1}(C_{n-2})+\gamma
_{k-1}(C_{n-3})$ for every $n\geq 4$ and $2 \leq k \leq n-2$ with
initial conditions $\gamma _{1}(C_{1})=1$, $\gamma _{1}(C_{2})=2$,
$\gamma _{1}(C_{3})=3$, $\gamma _{2}(C_{2})=1$ and $\gamma
_{2}(C_{3})=3$. Let us define the ordinary generating function for
the numbers $\gamma _{k}(C_{n})$ as
\begin{equation*}
H(x,y)=\sum_{n\geq 1}\sum_{k\geq 1}\gamma _{k}(C_{n})x^{n}\ y^{k}.
\end{equation*}

Analogously as we proved Theorem \ref{gf-paths}, we have the
following generating function for $\gamma _{k}(C_{n})$.

\begin{theorem}
\begin{equation*}
H(x,y)=\frac{x(1+2x+3x^{2})y}{1-(x+x^{2}+x^{3})y}.
\end{equation*}
\end{theorem}

We determine the formal power series of $H(x,y)$ expanded in powers
of $y$ to obtain
\begin{eqnarray*}
H(x,y) &=&\frac{x(1+2x+3x^{2})y}{1-(x+x^{2}+x^{3})y} \\
&=&x(1+2x+3x^{2})y\sum_{k\geq 0}(x+x^{2}+x^{3})^{k}y^{k} \\
&=&\sum_{k\geq 1}x(1+2x+3x^{2})(x+x^{2}+x^{3})^{k-1}y^{k} \\
&=&\sum_{k\geq 1}\left( \sum_{n\geq 1}\gamma
_{k}(C_{n})~x^{n}\right) y^{k}.
\end{eqnarray*}

Therefore, for every $k\geq 1$%
\begin{eqnarray}
\sum_{n\geq 1}\gamma _{k}(C_{n})~x^{n}
&=&x(1+2x+3x^{2})(x+x^{2}+x^{3})^{k-1}
\notag \\
&=&x^{k}(1+2x+3x^{2})(1+x+x^{2})^{k-1}  \label{poly-cycles}
\end{eqnarray}%
is a monic polynomial denoted by $h_{k}(x)$ of degree $3k$ such that%
\begin{equation*}
h_{k}(x)=\sum_{n=k}^{3k}\gamma
_{k}(C_{n})~x^{n}=\sum_{t=0}^{2k}\gamma _{k}(C_{k+t})~x^{k+t}
\end{equation*}%
since $\gamma _{k}(C_{n})=0$ if $n<k$ and the well-known fact that
$\gamma (C_{n})=\left\lceil \frac{n}{3}\right\rceil $ for every
$n\geq 1$. Observe that $h_{k}(x)$ has $2k+1$ terms as $g_{k}(x)$
does. In this case, polynomial $h_{k}(x)$ has no symmetry.

\begin{theorem} \label{Sol-cycles}
For every $k\geq 1$ and $t\geq 0$, the number of dominating $k$-sets
of the cycle $C_{k+t}$ is \label{gamma-k-cycles}
\begin{equation*}
\gamma _{k}(C_{k+t})=\sum_{m=0}^{\left\lfloor
\frac{t}{2}\right\rfloor +1} \binom{k-1}{t-m}\left(
\binom{t-m+2}{m+2}\binom{t-m}{m-2}\right) .
\end{equation*}
\end{theorem}

\begin{proof}
Using equalities (\ref{1xx2}) and (\ref{poly-cycles}), we have that
\begin{equation*}
h_{k}(x)=(x^{k}+2x^{k+1}+3x^{k+2})(1+x+x^{2})^{k-1}
\end{equation*}
\begin{eqnarray*}
&=&x^{k}\sum_{l=0}^{k-1}\binom{k-1}{l}\left( \sum_{m=0}^{l}\binom{l}{m}%
x^{m}\right) x^{l}+2x^{k+1}\sum_{l=0}^{k-1}\binom{k-1}{l}\left(
\sum_{m=0}^{l}\binom{l}{m}x^{m}\right) x^{l} \\
&&+3x^{k+2}\sum_{l=0}^{k-1}\binom{k-1}{l}\left( \sum_{m=0}^{l}\binom{l}{m}%
x^{m}\right) x^{l}
\end{eqnarray*}
\begin{eqnarray*}
&=&\sum_{l=0}^{k-1}\sum_{m=0}^{l}\binom{k-1}{l}\binom{l}{m}%
x^{k+l+m}+\sum_{l=0}^{k-1}\sum_{m=0}^{l}2\binom{k-1}{l}\binom{l}{m}%
x^{k+l+m+1} \\
&&+\sum_{l=0}^{k-1}\sum_{m=0}^{l}3\binom{k-1}{l}\binom{l}{m}x^{k+l+m+2}
\end{eqnarray*}
\begin{equation*}
=\sum_{l=0}^{k-1}\binom{k-1}{l}x^{k}\sum_{m=0}^{l}\left[ \binom{l}{m}%
x^{l+m}+2\binom{l}{m}x^{l+m+1}+3\binom{l}{m}x^{l+m+2}\right]
\end{equation*}
\begin{equation*}
=\sum_{l=0}^{k-1}\binom{k-1}{l}x^{k}\left[ \sum_{m=0}^{l+2}\left( \binom{l}{m%
}+2\binom{l}{m-1}+3\binom{l}{m-2}\right) x^{m+l}\right] .
\end{equation*}
Using Pascal's formula,
\begin{eqnarray*}
h_{k}(x) &=&x^{k}\sum_{l=0}^{k-1}\binom{k-1}{l}\left[
\sum_{m=0}^{l+2}\left(
\binom{l+2}{m}+2\binom{l}{m-2}\right) x^{m}\right] x^{l}\text{ } \\
&=&x^{k}\sum_{l=0}^{k-1}\sum_{m=0}^{l+2}\binom{k-1}{l}\left( \binom{l+2}{m}+2%
\binom{l}{m-2}\right) x^{l+m}
\end{eqnarray*}
\begin{eqnarray*}
&=&x^{k}\sum_{t=0}^{2k}\sum_{m=0}^{\left\lfloor \frac{t}{2}\right\rfloor +1}%
\binom{k-1}{t-m}\left( \binom{t-m+2}{m}+2\binom{t-m}{m-2}\right)
x^{t}\text{
} \\
&=&\sum_{t=0}^{2k}\sum_{m=0}^{\left\lfloor \frac{t}{2}\right\rfloor +1}%
\binom{k-1}{t-m}\left( \binom{t-m+2}{m}+2\binom{t-m}{m-2}\right)
x^{k+t}.
\end{eqnarray*}
Hence,
\begin{equation*}
\gamma _{k}(C_{k+t})=\sum_{m=0}^{\left\lfloor \frac{t}{2}\right\rfloor +1}%
\binom{k-1}{t-m}\left( \binom{t-m+2}{m}+2\binom{t-m}{m-2}\right) ,
\end{equation*}
the equality of the theorem.
\end{proof}

The explicit formula of the theorem generates sequence A212634 of
\cite{Sloane}.

\begin{cor}
Let $k$ be a positive integer. The number of the dominating sets of
minimum cardinality for cycles is
\begin{equation*}
\begin{array}{ll}
\qquad \gamma _{k}(C_{3k})=3 & \text{if }n=3k\text{,} \\
\gamma _{k+1}(C_{3k+1})=\binom{k+2}{2}+2\binom{k}{2}+3k & \text{if }n=3k+1%
\text{ and} \\
\gamma _{k+1}(C_{3k+2})=3k+2 & \text{if }n=3k+2\text{.}
\end{array}
\end{equation*}
\end{cor}

By Theorem \ref{gamma-k-cycles} and Corollary \ref{wheels}, we conclude that%
\begin{eqnarray*}
\gamma _{k}(W_{k+t}) &=&\binom{k+t-1}{t}+\gamma _{k}(C_{k-1+t}) \\
&=&\binom{k+t-1}{t}+\sum_{m=0}^{\left\lfloor
\frac{t}{2}\right\rfloor +1} \binom{k-2}{t-m}\left(
\binom{t-m+2}{m}+2\binom{t-m}{m-2}\right) ,
\end{eqnarray*}
for every $k\geq 1$ and $t\geq 0$ such that  $k+t \geq 4$. This
identity generates the terms of sequence A212635 of \cite{Sloane}.

\bigskip

\textbf{Acknowledgement. }This research is supported by Grant
CONACYT CB2012-01178910.

\end{document}